\documentclass{daj}

\dajAUTHORdetails{%
  title = {Larger Corner-Free Sets from Better NOF Exactly-$N$ Protocols}, %% please capitalize all significant words
  author = {Nati Linial and Adi Shraibman},
  plaintextauthor = {Nati Linial and Adi Shraibman},
  plaintexttitle = {Larger Corner-Free Sets from Better NOF Exactly-N Protocols}, %%  title without math or LaTeX
  keywords = {Corner-free sets, Communication complexity},
}   %%% END \dajAUTHORdetails

\dajEDITORdetails{%
   year={2021},
   number={19},
   received={2 February 2021},   % received date: example: 7 January 2017
   published={4 October 2021},  % published date
   doi={10.19086/da.28933},       % XXX = number of paper, e.g. da006 for paper#6
}   %%% END \dajEDITORdetails

\usepackage{amsmath, amsfonts}
\usepackage{comment}

\newtheorem{theorem}{Theorem}[section]

\newtheorem{claim}[theorem]{Claim}

\newtheorem{conjecture}[theorem]{Conjecture}

\newtheorem{proposition}[theorem]{Proposition}

\newtheorem{protocol}{Protocol}

\newcommand{\QED}{\hfill$\;\;\;\rule[0.1mm]{2mm}{2mm}$}

\newenvironment{proof}{\begin{pproof}}{\QED\end{pproof}~\\}

\newcommand{\rs}{Ruzsa-Szemer\'{e}di}

\newcommand{\trans}{\ensuremath{\mathbb{T}}}

\begin{document}

\begin{frontmatter}[classification=text]
%% EDITOR: this will force the keywords to appear right after the Abstract.
%%   If the abstract is too long and would force the keywords off the
%%   front page, please comment out % [classification=text] above
%%   This way the keywords will be floated on the bottom of the first page
%%   even though the Abstract spills over to the next page.

%%% AUTHOR: Title goes here.  This line is optional.  You must use it
%%   if title has footnote attached or requires nontrivial typesetting,
%%   e.g., inclusion of linebreaks to force nice layout.
\title{Larger Corner-Free Sets from Better NOF Exactly-$N$ Protocols} %% please capitalize all significant words

%%% AUTHOR:
%%% List all authors. If you wish, place grant acknowledgements in \thanks.
%%% In brackets include a short tag for each author.
\author[nati]{Nati Linial\thanks{Supported in part by Grant 659/18 "High dimensional combinatorics" of the Israel Science Foundation.}}
\author[adi]{Adi Shraibman}

%%% AUTHOR: Abstract goes here
\begin{abstract}
A subset of the integer planar grid $[N] \times [N]$ is called \emph{corner-free}
if it contains no triple of the form $(x,y), (x+\delta,y), (x,y+\delta)$. It is known that
such a set has a vanishingly small density, but how large this density can be remains unknown.
The only previous construction, and its variants, were based on Behrend's large subset of 
$[N]$ with no $3$-term arithmetic progression.
Here we provide the first construction of a corner-free set that does not rely on 
a large set of integers with no arithmetic progressions. 
Our approach to the problem is based on the theory of communication complexity.\\
In the $3$-players exactly-$N$ problem the players need to decide whether
$x+y+z=N$ for inputs $x,y,z$ and fixed $N$. This is the first problem considered in the 
multiplayer Number On the Forehead (NOF) model.
Despite the basic nature of this problem, no progress has been made on it
throughout the years. Only recently have explicit
protocols been found for the first time, yet no improvement in complexity has been achieved to date. The
present paper offers the first improved protocol for the exactly-$N$ problem.
\end{abstract}
\end{frontmatter}

\section{Introduction}

Van der Waerden's well known theorem \cite{van1927beweis} states that
for every $r, k$ and every large enough $N$, if the elements of $[N]:=\{1,\ldots,N\}$ 
are colored by $r$ colors, then there must exist a length-$k$ monochromatic arithmetic progression.
Erd\H os and Tur\'an introduced the density version of this theorem. Let $\rho_k(N)$
be the largest density of a subset of $[N]$ without an arithmetic progression of length $k$.
Szemer\'edi's famous theorem \cite{szemeredi1975sets} shows\footnote{Unless otherwise
specified, all asymptotic statements are taken with $N\to\infty$.} that $\rho_k(N) = o(1)$ for every $k\ge 3$.

Extending van der Waerden's theorem, Gallai proved that in every finite coloring of $\mathbb{Z}^2$
some color contains arbitrarily large square subarrays. In search of a density version of Gallai's theorem, 
Erd\H os and Graham asked about the largest density of a subset of the integer grid $[N]\times[N]$
without a {\em corner}, i.e., a triple $(x,y), (x+\delta,y), (x,y+\delta)$ for
some $\delta\neq 0$. Denote this quantity by $\rho^{\angle}_3(N)$.

Ajtai and Szemer\'edi \cite{ajtai1974sets} proved the first {\em corners theorem}, showing that $\rho^{\angle}_3(N) = o(1)$.
This theorem easily yields that $\rho_3(N) = o(1)$, namely, the $k=3$ case of Szemer\'edi's theorem,
due to Roth \cite{roth1953certain}. Later on
Solymosi \cite{Solymosi2003} showed how to derive Ajtai and Szemer\'edi's corners theorem
from the triangle removal lemma \cite{ruzsa1978triple}.

The quantitative aspects of all these results: Szemer\'edi's theorem, the corner theorem,
the $(6,3)$ theorem (e.g., \cite{ruzsa1978triple}) and the triangle removal lemma 
remain unfortunately poorly understood. 
The upper bounds have gradually improved over the years
and the current "world record" of Bloom and Sisask \cite{bloom2020breaking} is
\[\rho_3(N)\le(\log N)^{-1-c}\text{~~for some absolute constant~~} c>0.\]
In contrast, not much has happened
with lower bounds in these problems. Behrend \cite{behrend1946sets} has constructed a large subset of $[N]$ without 
a $3$-term arithmetic progression, whence\footnote{All logarithms in this paper are in base $2$.}
$$
\rho_3(N) \ge 2^{-2\sqrt{2}\sqrt{\log N}+o(\sqrt{\log N})}.
$$
Elkin's modification \cite{elkin2010improved} of Behrend's construction, has
improved only the little-o term.
Behrend's construction also yields the previously best known lower bound on $\rho^{\angle}_3(N)$, viz.

\begin{equation}\label{old_bound}
\rho^{\angle}_3(N) \ge 2^{-2\sqrt{2}\sqrt{\log N}+o(\sqrt{\log N})}= 2^{-2.828...\sqrt{\log N}+o(\sqrt{\log N})}.    
\end{equation}

Here we improve this bound as follows
\begin{theorem}
\label{lb:corner}
$$
\rho^{\angle}_3(N) \ge 2^{-2\sqrt{\log e}\sqrt{\log N}+o(\sqrt{\log N})}= 2^{-2.4022...\sqrt{\log N}+o(\sqrt{\log N})}.
$$
There is an explicit corner-free subset of $[N]\times[N]$ of size 
$$
N^2/2^{2\sqrt{\log e} \sqrt{\log N}+o(\sqrt{\log N})}.
$$
\end{theorem}

\subsection{The Computational Perspective}

The multiplayer Number On the Forehead (NOF) model of communication complexity
was introduced by Chandra, Furst and Lipton \cite{CFL83}. Given a function $f: [N]^{k}\to \{0,1\}$,
the $k$ players in this scenario should jointly find out  $f(x_1,\ldots,x_k)$. We think of
$x_i$ as being placed on player $i$'s forehead, so that each player sees the whole input bar one argument. 
Players communicate by writing bits on a shared blackboard according
to an agreed-upon protocol. This model is intimately connected to several
key problems in complexity theory. E.g., lower bounds on the size of $ACC^0$ 
circuits for a natural function in $P$ \cite{Yao90, HG91}, branching programs,
time-space tradeoffs for Turing machines \cite{KN97}, and proof complexity \cite{BPS06}. 
In addition, progress in the NOF model, even for specific problems and for $k=3$, would have profound implications 
in graph theory and combinatorics \cite{hdp17, alon2020number}.

Much of Chundra, Furst and Lipton's seminal paper \cite{CFL83} is dedicated to the exactly-$N$ function
$f: [N]^{k}\to \{0,1\}$, where $f(x_1,\ldots,x_k)=1$ iff $\sum x_i=N$. They discovered a connection between the communication 
complexity of this function and well-known problems in additive combinatorics and Ramsey theory. They
used Ramsey's theory to prove a (rather weak) lower 
bound on the NOF communication complexity of this function. Using the connection to additive number theory,
they showed that a $O(\sqrt{\log N})$ protocol exists, although they have not made this protocol explicit. 

Our main concern here is with the $3$-player NOF
exactly-$N$ problem (or, the essentially equivalent addition problem where the players need to decide whether $x+y=z$). 
The three players jointly design a communication protocol $\mathcal{CP}$.
Then they get separated, and are given access to inputs $x,y,z$, as described above. Namely,
player $P_x$ gets to see inputs $y$ and $z$, $P_y$ sees $x$ and $z$, and $P_z$
sees $x$ and $y$. According to
their chosen $\mathcal{CP}$, they take turns writing messages
on a blackboard that is visible to all three players. The game ends when every player can deduce
whether $x+y+z=N$ (resp.\ $x+y=z$) based on the inputs she sees, and the contents of the blackboard (called {\em transcript}). 
The {\em complexity} of $\mathcal{CP}$ is the maximal length of a transcript over all instances $x,y,z$. 
As a function of $N$, the {\em communication complexity} of the problem is the minimal complexity of a protocol 
$\mathcal{CP}$ that solves the problem correctly for all inputs.
A {\em one-round} protocol starts with $P_z$ who writes a message on the board.
Subsequently $P_x$ and $P_y$ each write a single verification bit.
Let us spell out the connection between the communication complexity of the $3$-players NOF
addition problem and the corners theorems: 
\begin{claim}[\cite{CFL83}, implicit]
\label{claim_CFL_1}
\begin{enumerate}
\item 
There is an optimal one-round protocol for exactly-$N$.
\item
Let $T=\trans(x,y)$ be the message that $P_z$ posts on inputs $(x,y)$ in a one-round protocol 
for exactly-$N$. Then the set
\[S(T) = \{(x,y): \trans(x,y) = T\}\]
is corner-free.
\end{enumerate}
\end{claim}
See \cite{CFL83, beigel2006multiparty, ada2015nof, hdp17, alon2020number} for more details about the above claim and the relation
between communication complexity and additive combinatorics. 

There are several reasons why it is highly significant to
determine the communication complexity of the exactly-$N$ function,
aside of the very fundamental nature of the problem:
\begin{itemize}

\item Our poor understanding of this question is
manifested by the huge gap between the upper and lower
bounds that we currently have on the communication complexity of this problem. This
gap is double exponential for three players, and even worse for $k>3$
players.

\item
Despite the significance of the NOF model, we still know very little
about it. The rich web of mathematical and computational concepts surrounding 
the exactly-$N$ function suggests that it may open the gate to 
progress in understanding numerous other NOF functions.

\item The $k$-player exactly-$N$ function is a {\em graph function} \cite{BDPW07}. For most functions in this class
the deterministic and randomized communication complexity differ substantially,
but no explicit function with deterministic complexity larger than polylogarithmic is presently known.   

\item This problem is {\em equivalent} to corner theorems in additive combinatorics (e.g., \cite{ajtai1974sets}),
and is closely related to other important problems such as constructing \rs\ graphs
and the triangle removal lemma \cite{hdp17, alon2020number}. 

\end{itemize} 

\begin{comment}
Nevertheless, progress on the complexity of the NOF exactly-$N$ problem has mostly been made
on the additive combinatorics side and includes several breakthrough
results such as Szemer\'{e}di's regularity lemma \cite{szemeredi1975sets}, and its extension to hypergraphs 
\cite{gowers2007hypergraph, rodl2004regularity, nagle2006counting}. 
Translated back to the realm of NOF communication complexity, these advances
bear on lower bounds in communication complexity, yet there is essentially nothing concerning upper bounds.   
We believe that the more promising line of attack is for advances in communication complexity to 
shed light on questions in additive combinatorics by
exploiting the power of new algorithmic ideas.
\end{comment}

As mentioned, the existence of a protocol for the exactly-$N$ problem has already been known since \cite{CFL83}.
However, this was just an existential statement and no actual protocol was provided. This lacuna
was recently remedied with two protocols \cite{hdp17, alon2020number} of the exact same complexity
as the one whose existence was proven in \cite{CFL83}
\begin{equation}\label{old_ub}
2\sqrt{2}\sqrt{\log N}+o(\sqrt{\log N})= 2.828...\sqrt{\log N}+o(\sqrt{\log N}).
\end{equation}

Here we give the first improved protocol for the exactly-$N$ problem,
and prove
\begin{theorem}
\label{main_p}
There is an explicit protocol for $3$-players exactly-$N$ of complexity 
\begin{equation}\label{better_upper_bound}
2\sqrt{\log e} \sqrt{\log N}+o(\sqrt{\log N}) = 2.4022...\sqrt{\log N}+o(\sqrt{\log N}).
\end{equation}

\end{theorem}

\section{Proof of Theorem \ref{main_p}}

We recall that the three players called $P_x, P_y$ and $P_z$ get to see the inputs $(y,z), (x,z)$ and $(x,y)$ respectively.
Given integers $q,d > 1$, define $g=g_{q, d}(\alpha, \beta, \gamma)$ to be $1$ if $\alpha+\beta = \gamma$ and $0$
otherwise. Here $\alpha, \beta\in [q]^d$, $\gamma\in [2q]^d$ and addition is
vector addition in $\mathbb{R}^d$. The following one-round protocol \cite{alon2020number} for $g$ is 
correct because the inequality
$\|2\alpha-\gamma\|^2+\|2\beta-\gamma\|^2\ge 2\|\alpha-\beta\|^2$ holds always
and is an equality iff $\gamma=\alpha+\beta$.

\bigskip
\fbox{
\begin{minipage}{5.8in}
\begin{protocol}{\sf{A protocol for $g_{q, d}$}}
\label{base_protocol}
\begin{enumerate}

\item $P_z$ computes $\|\alpha-\beta\|_2^2$, and writes the result on the board.

\item $P_y$ writes $1$ iff $\|\alpha-\beta\|_2^2 = \|2\alpha-\gamma\|_2^2$.

\item $P_x$ writes $1$ iff $\|\alpha-\beta\|_2^2 = \|2\beta-\gamma\|_2^2$.

\end{enumerate}

\end{protocol}
\end{minipage}
}
\bigskip

The cost of this protocol is $2+\log dq^2$.

The above is an efficient method to decide high-dimensional vector addition, but our objective is to decide the
integer addition relation $X+Y+Z=N$. We let $x=X, y=Y$ and $z=N-Z$, so the relation
we need to consider is $x+y=z$. 

Our protocol to decide whether $x+y=z$ builds on the protocol for $g_{q, d}$. It is the issue of carry in integer
addition that makes this decision problem harder.
The integers $q,d > 1$ are chosen so that
\begin{equation}\label{qnd}
 2qN > q^d \ge 2N.   
\end{equation}
the specific choice is made below so as to minimize the cost of the protocol.\\ 
We denote by $w_q$ the vector that corresponds to
the base $q$ representation of the integer $w$. 

As usual, $e_i$ is the $d$-dimensional vector with $1$ in the $i$-th coordinate and zeros elsewhere. 
Let $C(x,y) \in \{0,1\}^d$ be the carry vector when $x$ and $y$ are added in base $q$.
The relation $x+y=z$ among integers is equivalent to the vector relation \[x_q+y_q=\zeta,\] where
the $i$-th coordinate of $\zeta$ is
\[\zeta_i=z_i+q\cdot C(x,y)_i-C(x,y)_{i-1}\]
(Here $C(x,y)_0=0$).
The protocol from \cite{alon2020number} now suggests itself: $P_z$ posts $C(x,y)$,
and Protocol \ref{base_protocol} is used to decide the relation $x_q+y_q=\zeta$. 
This yields again the estimate (\ref{old_ub}).

The alternative approach that we adopt here considers instead the equivalent vector relation
\[x_q+\eta=z_q\]
where
\[\eta=(x+y)_q-x_q.\]
Concretely, the $i$-th coordinate of $\eta$ is:
\[\eta_i=y_i-q\cdot C(x,y)_i+C(x,y)_{i-1}.\]
In order to run Protocol \ref{base_protocol}, $P_z$ needs to know $\eta$ and $x_q$, and he does. 
With $P_y$ it's even simpler, since he needs to know $x_q$ and $z_q$ which are his inputs. The
only difficulty is with $P_x$ who needs to know $z_q$ (which he does) and $\eta$. The latter
is not part of his input and $P_z$ fills in the missing information for him.

The obvious solution is for $P_z$ to reveal $C(x,y)$ to $P_x$ using $d$ bits of information.
However, we can save communication by exploiting the fact
that $P_x$ and $P_z$ share some information, i.e., they both know $y$ for every $y\neq 0$.

By a standard argument in this area which we detail below (Proposition \ref{smear}), a protocol that works for {\it typical}
pairs $x,y$ can be easily modified to work in {\it all} cases. 
So, let us pick $x$ and $y$ uniformly at random from among the $d$-digit numbers in base $q$
and think of $C=C(x,y)$, the vector of carry bits as a random variable on this probability space.
The number of bits that $P_z$ needs to post so that $P_x$ gets to know $C$,
and therefore know $\eta$, is $H(C|y)$,
the entropy of $C$ given $y$. The gain is clear, since $H(C)>H(C|y)$.

It remains to estimate $H(C|y)$. Let $X$ be the random variable that is
a uniformly sampled subset of $[s]$ of cardinality $\ge t$, for some
integers $s\ge t\ge 0$. We recall that $H(X)=(1+o_s(1))\cdot s\cdot h(t/s)$,
where $h(\cdot)$ is the univariate entropy function, and the same holds also if we 
consider subsets of $[s]$ of cardinality $\le t$. 
Let $r$ be an integer in the range $d\gg r\gg 1$, 
e.g., $r\approx\sqrt{d}$. For $j=1,\dots,r$, let
\[
S_j= \{i~|~ \frac{qj}{r}>y_i\ge\frac{q(j-1)}{r}\},
\]
where $q>y_i\ge 0$ is the $i$-th digit of $y$. A carry occurs in digit $i\in S_j$
only if $x_i> \frac{q(r-j)}{r}$, where $x_i$ is the $i$-th digit of $x$. Then
\[
H(C|y)\le (1+o_r(1))\sum_{j=1}^r \frac{|S_j|}{d}h(\frac{j}{r}).
\]
Since $y$ is chosen at random, $|S_j|\le(1+o_r(1))\frac{d}{r}$, and so
\[
H(C|y)\le (1+o_r(1))\sum_{j=1}^r \frac{1}{r}h(\frac{j}{r}).
\]
The limit of this expression as $r\to\infty$ is
\[
\lambda=\int_0^1 h(u)du = \frac{\log e}{2}=0.721...
\]

It is left to optimize on $q$ and $d$. The complexity of our protocol is
\[
\lambda d+\log dq^2 + 2,
\]
where recall that $2qN > q^d \ge 2N$. It is not hard to verify that choosing
\begin{equation}\label{d_and_q}
d=\sqrt{\frac{2}{\lambda} \log 2N}~~~~~
q=2^{\sqrt{\frac{\lambda}{2}\log 2N}},
\end{equation}
we get a protocol with complexity
\[
2\sqrt{2\lambda \log N} +o(\sqrt{\log N}),
\]
and this is asymptotically optimal in our setting.

To sum up, here is the protocol which proves Theorem \ref{main_p}:

\bigskip
\fbox{
\begin{minipage}{5.8in}
\begin{protocol}{\sf{A protocol for exactly-$N$, for typical
pairs $x,y$}}\label{full_prtcl}
\begin{enumerate}

For $d,q$ as in Equation (\ref{d_and_q})

\item $P_z$ publishes the vector $\eta = (x+y)_q-x_q$ in a way that $P_x$ can read.

\item The players run Protocol~\ref{base_protocol} for $g_{q,d}$ on inputs
$x_q, \eta, z_q$. That is:

\begin{enumerate}

\item $P_z$ writes $\|\eta-x_q\|_2^2$ on the board

\item $P_y$ writes $1$ iff $\|\eta-x_q\|_2^2 = \|2x_q-z_q\|_2^2$.

\item $P_x$ writes $1$ iff $\|\eta-x_q\|_2^2 = \|2\eta-z_q\|_2^2$.

\end{enumerate}

\end{enumerate}

\end{protocol}
\end{minipage}
}
\bigskip

\begin{proposition}\label{smear}
Let $\cal P$ be an NOF protocol for the exactly-$N$ that works correctly
for an $\Omega(1)$-fraction of the input pairs $x,y$ (and every $z$) with communication
complexity $\Phi(N)$. Then there is an NOF protocol that works for all
inputs with communication complexity $\Phi(N) + O(\log\log N)$.
\end{proposition}

\begin{proof}
Let $S \subseteq [N]\times [N]$ be the set of input pairs $x,y$ on which $\cal P$ succeeds.
We claim that there is a collection $F$ of $O(\log N)$ vectors $\Delta\in [N]\times [N]$
such that 
\[\cup_{\Delta\in F} (S+\Delta) \supseteq [N]\times [N].\]
In the modified protocol, $P_z$ sees $x, y$ and announces the index of some $\Delta=(\Delta_1,\Delta_2)\in F$
for which $(x-\Delta_1, y-\Delta_2)\in S$. Then the players run Protocol \ref{full_prtcl}
with inputs $(x-\Delta_1, y-\Delta_2, z-\Delta_1-\Delta_2)$. 

The construction of $F$ uses a standard fact about the set-cover problem. For
a family of finite sets ${\cal X}\subseteq 2^{\Omega}$ we denote by $c(\cal X)$ the least number
of members in $\cal X$ whose union is $\Omega$. Also $c^{\ast}(\cal X)$ is the
minimum cost of a fractional cover. Namely,
\[
c^{\ast}(\cal X)\rm=\min \sum_{{\cal X}}\omega_X,
\text{~where~} \omega_X\ge 0 \text{~for every~} X\in{\cal X} \text{~and~} \sum_{x\in X}\omega_X\ge 1 \text{~for every~} x\in\Omega.
\]
Then \[c({\cal X})\le \log(|\Omega|)\cdot c^*(\cal X)\] (e.g., Lov\'asz \cite{Lov75})
and actually the greedy algorithm yields a set cover that meets this bound. 

In our case
$\Omega=[N]\times[N]$, and \[{\cal X}=\{(S+\Delta)\cap([N]\times[N])~|~\Delta\in [-N,N]\times [-N,N]\}.\]
It is easily verified that the weights $\omega_x=\frac{10}{N^2}$ constitute a fractional cover, so that $c^*({\cal X})\le 40$
and hence $c({\cal X})\le 80\log N$, as claimed.
\end{proof}

\section{Discussion}

The strong relation between the exactly-$N$ problem in the NOF model and questions in additive combinatorics has
been discovered decades ago, in the seminal paper of Chundra, Furst and Lipton \cite{CFL83}. 
However, this subject remains under-developed.
We believe that there is a lot to be done here, and many interesting avenues of research that this
study can take. One obvious candidate is to seek further improvement is the addition problem. 
We conjecture:

\begin{conjecture}
The NOF communication complexity of exactly-$N$ is $o(\sqrt{\log N})$.
Possibly it is much smaller, even as small as $(\log\log N)^{O(1)}$.
\end{conjecture}
In the realm of additive combinatorics these conjectures translate to:
\begin{conjecture}
$$
\rho^{\angle}_3(N) \ge 2^{-o(\sqrt{\log N})}.
$$
and possibly even
$$
\rho^{\angle}_3(N) \ge 2^{-(\log\log N)^{O(1)}}.
$$
\end{conjecture}

\paragraph{Remark}
Very shortly after seeing our arXiv posting, Ben Green was able to improve our upper bound on $\rho^{\angle}_3(N)$ even further 
\cite{green2021lower}. 
What is more, his argument makes no reference to communication complexity.
This begs the question: Is the communication complexity
perspective really useful ? With only 
limited evidence on our hands, we can only state our opinion and impression.
We believe that communication complexity can become a significant source of ideas 
for additive combinatorics. At least to us, this point of view has
revealed things that we found hard to notice otherwise.
We hope that convincing evidence for this belief will be found in the future.

%%% AUTHOR: optional acknowledgments here
\section*{Acknowledgments} %%  you may comment this out if no Ackno
We thank Noga Alon, Aya Bernstine and Alex Samorodnitsky for insightful discussions.

%%% AUTHOR:
%%% Bibliography goes here. Note that the arXiv cannot process bibtex
%%% or biber bibliographies.  Example of acceptable bibliograpy format:
\bibliographystyle{amsplain}

%% AUTHOR: You can generate such a bibliography from a .bib file by 
%% running pdflatex/bibtex/pdflatex/pdflatex and then pasting the .bbl file
%% between \begin{thebibliography} and \end{bibliography}

%%% AUTHOR: Include a short description of each author following the
%%% structure below. Use the same short tags used previously.  
%%% Use \imageat{} and \imagedot{} instead of "@" and "." in
%%% email addresses-this replaces the symbols with graphics to avoid 
%%% e-mail address harvesting from the .pdf file
\begin{dajauthors}
\begin{authorinfo}[nati]
  Nati Linial\\
  Hebrew University of Jerusalem\\
  Jerusalem, Israel\\
  nati\imageat{}cs\imagedot{}huji\imagedot{}ac\imagedot{}il\\
  \url{https://www.cs.huji.ac.il/~nati/}
\end{authorinfo}
\begin{authorinfo}[johan]
  Adi Shraibman\\
  The Academic College of Tel-Aviv-Yaffo\\
  Tel-Aviv, Israel\\
  adish\imageat{}mta\imagedot{}ac\imagedot{}il \\
  \url{https://www2.mta.ac.il/~adish/}
\end{authorinfo}
\end{dajauthors}

\end{document}